\documentclass[11pt]{amsart}

\usepackage[a4paper,hmargin=3.5cm,vmargin=4cm]{geometry}
\usepackage{amsfonts,amssymb,amscd,amstext}
\usepackage{graphicx}
\usepackage{color}
\usepackage[dvips]{epsfig}

\usepackage[utf8]{inputenc}
\usepackage{hyperref}

\usepackage{verbatim}
\usepackage{fancyhdr}
\input xy
\xyoption{all}

\usepackage{times}
\usepackage{enumerate}
\usepackage{titlesec}
\usepackage{mathrsfs}

\titleformat{\section}
{\filcenter\bfseries\large} {\thesection{.}}{0.2cm}{}
\titleformat{\subsection}[runin]
{\bfseries} {\thesubsection{.}}{0.15cm}{}[.]
\titleformat{\subsubsection}[runin]
{\em}{\thesubsubsection{.}}{0.15cm}{}[.]

\usepackage[up,bf]{caption}

\newtheorem{theorem}{Theorem}

\newtheorem{lemma}{Lemma}

\theoremstyle{definition}

\newtheorem{remark}{Remark}

\numberwithin{equation}{section}
\numberwithin{figure}{section}

\newcommand\Fscr{\mathscr{F}}

\newcommand\Gscr{\mathscr{G}}

\newcommand\B{\mathbb{B}}
\newcommand\C{\mathbb{C}}
\newcommand\D{\overline{\mathbb D}}

\renewcommand\D{\mathbb D}

\newcommand\N{\mathbb{N}}
\renewcommand\P{\mathbb{P}}

\renewcommand\b{\mathbb{B}}

\newcommand\cd{\overline{\mathbb D}}

\newcommand\igot{\mathfrak{i}}

\renewcommand\igot{\mathfrak{i}}

\renewcommand\imath{\igot}

\newcommand\Id{\mathrm{Id}}

\newcommand\Aut{\mathrm{Aut}}

\begin{document}

\fancyhead[LO]{A foliation of the ball by complete holomorphic discs}
\fancyhead[RE]{A.\ Alarc\'on and F.\ Forstneri\v c} 
\fancyhead[RO,LE]{\thepage}

\thispagestyle{empty}

\begin{center}
{\bf\LARGE A foliation of the ball by complete holomorphic discs}

\vspace*{0.5cm}

{\large\bf  Antonio Alarc{\'o}n \; and \; Franc Forstneri{\v c}} 
\end{center}


\vspace*{0.5cm}

\begin{quote}
{\small
\noindent {\bf Abstract}\hspace*{0.1cm}
We show that the open unit ball $\B^n$ of $\C^n$ $(n>1)$ 
admits a nonsingular holomorphic foliation by complete properly embedded holomorphic discs.

\vspace*{0.2cm}

\noindent{\bf Keywords}\hspace*{0.1cm} Riemann surface, holomorphic disc, foliation, complete Riemannian manifold

\vspace*{0.1cm}

\noindent{\bf MSC (2010):}\hspace*{0.1cm} 32B15, 32H02, 32M17, 53C12}
%
%
%
\end{quote}


\vspace{1mm}

\section{Introduction} 
\label{sec:intro}

An open connected submanifold $M$ of a Euclidean space is said to be {\em complete} if 
every divergent path in $M$ has infinite Euclidean length; equivalently, the restriction of the 
Euclidean metric $ds^2$ to $M$ is a complete Riemannian metric on $M$. Recall that a path 
$\gamma:[0,1)\to M$ is called divergent if $\gamma(t)$ leaves any compact subset 
of $M$ as $t\to 1$.

For $n>1$, we denote by $\B^n$ the open unit ball of $\C^n$. In this paper, we prove the following result.

%
%
\begin{theorem}\label{th:th1}
For any integer $n>1$ there exists a nonsingular holomorphic foliation $\Fscr$ of $\B^n$ all of
whose leaves are complete properly embedded holomorphic discs in $\B^n$.
\end{theorem}

Theorem \ref{th:th1} seems to be the first result in which one controls
the topology of all leaves in a nonsingular holomorphic foliation of the ball by
complete leaves; in our examples all leaves are the simplest possible ones, namely, discs. 
We do not know whether a comparable result holds with leaves that have prescribed nontrivial topology.

Before proceeding, we place our result in the context of what is known.

The question whether there exist bounded (relatively compact) complete 
complex submanifolds of $\C^n$ for $n>1$ was raised by P.\ Yang \cite{Yang1977} in 1977.
The first examples were found in 1979 by P.\ Jones \cite{Jones1979}
who showed that the disc $\D=\{z\in\C:|z|<1\}$ admits a complete bounded holomorphic immersion
into $\C^2$, embedding into $\C^3$, and proper embedding into the ball of $\C^4$.
Interest in this subject has recently been revived due to new construction methods.
It was shown that there are complete properly {\em immersed} holomorphic curves 
in $\B^2$, and embedded ones in $\B^3$, with any given topology \cite{AlarconLopez2013MA}, 
and also those with the complex structure of any given bordered Riemann surface 
\cite{AlarconForstneric2013MA,AlarconForstneric2014IM}. A related result in higher
dimension was obtained by Drinovec Drnov\v sek \cite{Drinovec2015JMAA}.
Parallel developments were made in minimal surface theory 
where the corresponding circle of questions is known as the {\em Calabi-Yau problem}; 
see the survey \cite{AlarconForstnericJAMS} and the preprint \cite{AlarconForstneric2019CY}.

It is a considerably more challenging task to construct complete properly {\em embedded} 
holomorphic curves in $\B^2$ and, more generally, complete complex hypersurfaces in $\B^n$ for $n>1$. 
The first examples for $n=2$ were given by A.\ Alarc\'on and F.J.\ L\'opez in \cite{AlarconLopez2016JEMS}. 

In a pair of landmark papers \cite{Globevnik2015AM,Globevnik2016MA} in 2015--16, 
J.\ Globevnik constructed for any pair of integers $1\le k<n$ a complete $k$-dimensional complex 
submanifold of $\B^n$ and, more generally, of any pseudoconvex Runge domain in $\C^n$. 
For $k=n-1$ his construction provides a possibly singular holomorphic foliation of the ball $\B^n$ 
by complete complex hypersurfaces, most of which are smooth. 
Subsequently, Alarc\'on \cite{Alarcon2018} introduced to this subject Forstneri\v c's techniques from 
\cite{Forstneric2003AM}, concerning noncritical holomorphic functions, and showed that every smooth complex
hypersurface in the ball $\B^n$ is a leaf of a {\em nonsingular holomorphic foliation} of $\B^n$ by hypersurfaces 
such that all except perhaps the initial one are complete. An analogous result was established for 
foliations of any codimension. This provides both a converse to, and an extension of the aforementioned 
theorem of Globevnik. 

Foliations in \cite{Alarcon2018,Globevnik2015AM,Globevnik2016MA}
are given by level sets of suitable holomorphic functions on $\b^n$ (or, more generally, 
of submersions $\b^n\to\C^q$ with $1\le q<n$), so the topology of their leaves is not controlled. 
The same can be said about the examples in \cite{AlarconLopez2016JEMS}.
By a different technique, using holomorphic automorphisms of $\C^n$, 
Alarc\'on, Globevnik, and L\'opez \cite{AlarconGlobevnikLopez2016Crelle} obtained complete closed 
complex hypersurfaces in the ball $\B^n$ $(n>1)$ with certain restrictions on the topology of the examples, 
and with any given topology when $n=2$ \cite{AlarconGlobevnik2017}. It follows in particular that the disc 
$\D$ can be embedded as a complete proper holomorphic curve in $\B^2$. However, their results do 
not apply to foliations.

Our proof of Theorem \ref{th:th1} follows a similar approach as the one in 
\cite{AlarconGlobevnikLopez2016Crelle}, but with an addition which enables
us to control the topology and completeness of all leaves in a foliation, and 
not only of a single curve. By using holomorphic automorphisms we 
successively twist a holomorphic foliation of $\C^n$ by complex lines in 
order to make bigger and bigger parts of the foliation avoid pieces of a suitable labyrinth 
$\Gamma$ in $\B^n$. The labyrinth is a closed set in $\B^n$ exhausted by an increasing sequence
$K_j=\bigcup_{i=1}^j \Gamma_i$ $(j\in\N)$ 
of compact polynomially convex sets, where $\Gamma_i$ is contained in a
spherical shell $B_{i+1}\setminus B_{i}$ between two concentric balls in $\B^n$ 
and consists of finitely many pairwise disjoint closed round balls in suitably 
chosen affine real hyperplanes. The main property of the labyrinth is that any divergent curve
in $\B^n$ which avoids the set $\bigcup_{k=i}^\infty \Gamma_k$ 
for some $i\in\N$ has infinite length. Such labyrinths have been constructed in 
\cite{AlarconGlobevnikLopez2016Crelle,Globevnik2015AM}.
Note that each connected component of the intersection of $\B^n$ with an embedded
complex line $L\subset \C^n$ is Runge in $L$ and hence simply connected;
since it is also bounded, it is a properly embedded disc in $\B^n$. 
Our construction therefore yields a sequence of foliations $\{\Fscr_i\}_{i\in\N}$ of $\B^n$ by discs
such that all leaves of $\Fscr_i$ intersecting a compact subset $B_i\subset \B^n$
have intrinsic diameter bigger than a certain number $k_i$, with $k_i\to+\infty$
and $B_i$ increasing to $\B^n$ as $i\to \infty$. 
In the limit foliation $\Fscr$, all leaves are discs with infinite intrinsic diameter, hence complete.

\section{The construction}\label{sec:proof}

Fix an integer $n>1$. Denote by $\Aut(\C^n)$ the group of holomorphic automorphisms of $\C^n$.
We shall need the following result concerning moving compact convex sets in $\C^n$
by holomorphic automorphisms; see \cite[Theorem 2.3]{ForstnericRosay1993}
or \cite[Corollary 4.12.4, p.\ 158]{Forstneric2017E} for more general statements.

%
%
%
\begin{lemma}\label{lem:starshaped}      
Let $K_0,K_1,\ldots, K_m$ be pairwise disjoint compact convex sets in $\C^n$ and let 
$\Psi_j \in \Aut(\C^n)$ $(j=0,1,\ldots,m)$ be such that the images 
$K'_j =\Psi_j(K_j)$ are pairwise disjoint. If the sets $K=\bigcup_{j=0}^m K_j$ and 
$K'=\bigcup_{j=0}^m K'_j$ are polynomially convex, then for any 
$\delta>0$ there exists $\Psi\in \Aut(\C^n)$ such that
\begin{equation}\label{eq:Psi}
	|\Psi(z)-\Psi_j(z)| < \delta\quad \text{for all $z\in K_j$, \ $j=0,1,\ldots,m$.}
\end{equation}
\end{lemma}

The following lemma provides the induction step in the proof of Theorem \ref{th:th1}.

%
%
\begin{lemma}\label{lem:main}
Let $B$ be a compact convex set contained in the ball $\B^n\subset \C^n$,
and let $\Gamma=\bigcup_{j=1}^m \Gamma_j \subset \B^n\setminus B$ 
be a union of finitely many, pairwise disjoint, compact convex sets $\Gamma_j$
such that the set $B \cup \Gamma$ is polynomially convex.
If $\Phi\in\Aut(\C^n)$, then for any numbers $r>0$ and $\epsilon>0$ 
there exists $\Theta\in\Aut(\C^n)$ such that 
\begin{enumerate}[\rm (a)]
\item $\Theta(\Phi(r\cd\times \C^{n-1}))\cap \Gamma=\varnothing$, and 
\smallskip
\item $|\Theta(z)-z|<\epsilon$ for all $z\in B$.
\end{enumerate}
\end{lemma}

\begin{proof}
Let $K_0$ be a compact convex neighbourhood of $B$, and for each $j=1,\ldots,m$ let $K_j$ be 
a compact convex neighbourhood of $\Gamma_j$ such that the sets $K_0,\ldots,K_m$ are pairwise disjoint 
and their union $K=\bigcup_{j=0}^m K_j$ is polynomially convex.
(We refer to Stout \cite{Stout2007} for general results on polynomial convexity.)

Let $\Psi_0=\Id\in\Aut(\C^n)$ be the identity automorphism. For $j=1,\ldots,m$ we choose automorphisms 
$\Psi_j\in \Aut(\C^n)$ such that the compact sets $K'_j:=\Psi_j(K_j)$ $(j=0,1,\ldots,m)$ 
are pairwise disjoint, we have that
\begin{equation}\label{eq:disjoint1}
	K'_j \cap \Phi(r\cd\times\C^{n-1})=\varnothing\quad \text{for $j=1,\ldots,m$},
\end{equation}
and the union $\bigcup_{j=0}^m K'_j$ is polynomially convex. 
Clearly, such $\Psi_j$ exist: noting that $K_0'=K_0$, we may squeeze each convex set $K_j$ $(j=1,\ldots,m)$ 
by a dilation into a very small neighbourhood of an interior point of itself and then translate 
their images into sufficiently small pairwise disjoint balls around some points in the complement of 
$K_0\cup \Phi(r\cd\times\C^{n-1})$. (We refer to \cite[proof of Corollary 4.12.4]{Forstneric2017E}
for the details.)  

Now, Lemma \ref{lem:starshaped} furnishes for every $\delta>0$ an
automorphism $\Psi\in\Aut(\C^n)$ satisfying \eqref{eq:Psi}. 
Let $\Theta=\Psi^{-1}$. If $\delta>0$ is small enough then condition (b) holds,
and we have $\Psi(\Gamma_j) \subset K'_j$ and hence $\Gamma_j\subset \Theta(K'_j)$ 
for every $j=1,\ldots,m$, which yields (a). Indeed, if $\Theta(\Phi(z))\in \Gamma_j$ 
for some $z\in r\cd\times \C^{n-1}$ then $\Phi(z)\in K'_j$ which contradicts \eqref{eq:disjoint1}. 
\end{proof}

%
%
\begin{proof}[Proof of Theorem \ref{th:th1}]
Fix an integer $n>1$. We exhaust the unit ball $\B^n\subset\C^n$ by an increasing sequence of closed balls 
\begin{equation}\label{eq:cupB}
	B_1\subset B_2\subset \cdots\subset \bigcup_{i=1}^\infty B_i = \B^n
\end{equation}
centered at the origin such that each $B_i$ is contained in the interior of the next ball $B_{i+1}$. 
Denote by $\rho_i$ the radius of $B_i$, so we have 
$0<\rho_1<\rho_2<\cdots<1$ with $\lim_{i\to\infty}\rho_i=1$. 

In each spherical shell $\mathring B_{i+1}\setminus  B_i$ $(i\in\N)$ we choose a compact set 
$\Gamma_i= \bigcup_{j=1}^{m_i}\Gamma_{i,j}$ consisting of finitely many, pairwise disjoint, compact 
convex sets $\Gamma_{i,j}$ and satisfying the following conditions. 
\begin{enumerate}[\rm (A)]
\item The set $B_i\cup \Gamma_i$ is polynomially convex for every $i\in\N$.
\smallskip
\item Every divergent path in $\B^n$ avoiding  $\Gamma^i =\bigcup_{k=i}^\infty \Gamma_k$ 
for some $i\in\N$ has infinite length.
\end{enumerate}
Labyrinths with these properties have been constructed in 
\cite{AlarconGlobevnikLopez2016Crelle,Globevnik2015AM}. In \cite{AlarconGlobevnikLopez2016Crelle} 
the connected components $\Gamma_{i,j}$ of $\Gamma$ are balls in suitably chosen affine 
real hyperplanes in $\C^n$.

We now describe the induction leading to the proof of Theorem \ref{th:th1}.

Recall that $\D=\{z\in\C:|z|<1\}$. Let $\P=\D^{n-1}\subset\C^{n-1}$ denote the unit 
$(n-1)$-dimensional polydisc. By $r\P$ for $r>0$ we denote the polydisc 
with polyradius $r$. Choose a number $\epsilon_0>0$ and set $r_0=0$, $B_0=\Gamma_0=\varnothing$. 
Let $\Phi_0=\phi_0=\Id\in\Aut(\C^n)$ be the identity map. 

We shall inductively find sequences $r_i>0$, $\epsilon_i>0$, and 
$\phi_i\in\Aut(\C^n)$ such that, setting $\Phi_i=\phi_i\circ \cdots\circ \phi_1$, the following 
conditions hold for every $i\in\N$.
\begin{enumerate}[\rm (a{$_i$})]
\item $r_i>r_{i-1}+1$ and $B_i \subset \Phi_{i-1}(r_i \P\times \C)$.
\smallskip
\item $|\phi_i(z)-z|<\epsilon_i$ for all $z\in B_i$.
\smallskip
\item $\Phi_i(r_j \overline{\P} \times \C)\cap \Gamma_j=\varnothing$ for $j=1,\ldots,i$. 
\smallskip
\item $0<\epsilon_{i}<\min\{ \epsilon_{i-1}/2, \rho_{i+1} -  \rho_{i}\}$.
\smallskip
\item For every  holomorphic map $\theta:B_i\to\C^n$ satisfying $|\theta(z)-z|<\epsilon_i$ for all $z\in B_i$
we have that $\theta(\Phi_{i-1}(r_j \overline{\P} \times \C)) \cap \Gamma_j=\varnothing$ for $j=1,\ldots,i-1$.
\end{enumerate}
Assume inductively that for some $i\in\N$ we have already found these objects
for all indices up to $i-1$; this trivially holds for $i=1$. 

Choose a number $r_i$ satisfying (a$_i$). Next, choose $\epsilon_{i}>0$ so small that
(d$_i$) and (e$_i$) are satisfied. When $i=1$, condition (e$_1$) is vacuous, while for $i>1$ it can be satisfied
by (c$_{i-1}$); note that for $j<i$ the set $\Gamma_j$ is contained in the interior of $B_i$,  
and the set $\Phi_{i-1}(r_j \overline{\P} \times \C) \cap B_{i}$ is compact and disjoint from $\Gamma_j$.

By property {\rm (A)} of the labyrinth, we may apply Lemma \ref{lem:main} with $\Phi=\Phi_{i-1}$ 
and obtain an automorphism $\phi_i\in\Aut(\C^n)$ (called $\Theta$ in the lemma) satisfying the 
approximation condition (b$_i$) and such that the automorphism $\Phi_i:=\phi_i\circ\Phi_{i-1}\in\Aut(\C^n)$ 
satisfies (c$_i$). (The lemma directly ensures that $\Phi_i$ satisfies condition (c$_i$) for $j=i$; it then also 
satisfies the same condition for indices $j=1,\ldots,i-1$ in view of the condition (e$_i$) on the number $\epsilon_i$.)
Thus, the induction may proceed.

In view of \eqref{eq:cupB} and conditions (b$_i$) and (d$_i$), we see from 
\cite[Proposition 4.4.1]{Forstneric2017E}  that the sequence $\Phi_i\in\Aut(\C^n)$ converges uniformly 
on compacts in the domain 
\[
	\Omega = \bigcup_{i=1}^\infty \Phi_i^{-1}(B_i) \subset \C^n 
\]
to a biholomorphic map $\Phi=\lim_{i\to\infty} \Phi_i : \Omega\to \B^n$. Moreover,
$\Phi_i^{-1}(B_i)$ for $i=1,2,\ldots$ is an increasing sequence of compact polynomially convex sets 
exhausting $\Omega$. It follows that $\Omega$ is a pseudoconvex Runge domain in $\C^n$. 
Conditions (b$_i$) and (d$_i$) also show that for any $z \in B_i$ and $k>i$ we have 
\[
	|\Phi_k\circ\Phi_{i}^{-1}(z)-z| = |\phi_k\circ\cdots \circ\phi_{i+1}(z)-z|
	< \sum_{j=i+1}^k \epsilon_j < \epsilon_i.
\]
Passing to the limit as $k\to\infty$ gives 
\begin{equation}\label{eq:estPhi}
	|\Phi\circ\Phi_{i}^{-1}(z)-z| <  \epsilon_i, \quad z\in B_i.
\end{equation}
Writing $\Phi=(\Phi\circ\Phi_i^{-1}) \circ\Phi_i=\theta\circ\Phi_i$, we see from \eqref{eq:estPhi} and (e$_{i+1}$) that 
\begin{equation}\label{eq:avoiding}
	\Phi\bigl((r_i \overline{\P} \times \C) \cap\Omega\bigr) \cap \Gamma_i=\varnothing,\quad i=1,2,\ldots.
\end{equation}

Write $z=(z',z_n)\in\C^n$ where $z'=(z_1,\ldots,z_{n-1})$. Let $\Gscr$ be the foliation of $\Omega$ by 
the connected components of the sets $(\{z'=c\}\times \C)\cap \Omega$ for $c\in \C^{n-1}$, and let 
$\Fscr=\Phi(\Gscr)$ be the image foliation of $\B^n$. Since $\Omega$ is
pseudoconvex and Runge in $\C^n$, the leaves of $\Gscr$ are discs or complex lines
which are proper in $\Omega$; hence the analogous condition holds for the leaves of $\Fscr$
in $\B^n$. Since $\B^n$ is bounded, all leaves of $\Fscr$ (and hence of $\Gscr$) are discs.

It remains to show that all leaves of $\Fscr$ are complete. Let $F\in \Fscr$. 
Fix a point $w=(w',w_n)\in G:=\Phi^{-1}(F)$. Note that $G$ is a disc in the line 
$L=\{(w',\zeta): \zeta \in\C\}$. Choose $i_0\in\N$ so large that $w \in r_{i_0}\P\times \C$ and $\Phi(w)\in B_{i_0}$. 
Clearly these conditions persist if we increase $i_0$. Since $\Phi_i\to \Phi$ uniformly on compacts 
in $\Omega$ as $i\to\infty$, we can increase $i_0$ if necessary so that we also have 
$\Phi_{i_0-1}(w)\in B_{i_0}$. Since $L=\{w'\}\times \C \subset r_{i_0}\P\times \C$ by the choice of $i_0$,
we see from \eqref{eq:avoiding} that 
$F\cap\Gamma_i \subset \Phi(L\cap \Omega)\cap \Gamma_i =\varnothing$ for all $i\ge i_0$,
and hence $F\cap \Gamma^{i_0}=\varnothing$. Since $F$ is proper in $\b^n$ and in view of the 
property (B) of the labyrinth, it follows that $F$ is complete.
\end{proof}

%
%
\begin{remark}\label{rem:Runge}
It might be possible to extend our proof to any pseudoconvex Runge domain
$\Omega\subset \C^n$ in place of the ball. The main problem is that we do not know how to construct
suitable labyrinths in such domains. One of the main requirements in our construction
is that the connected components of the labyrinth are convex or, more generally,
holomorphically contractible compact sets. In particular, we are unable to use labyrinths
constructed by embedding $\Omega$ as a closed complex submanifold $X\subset \C^N$ 
for a suitable $N$ and taking labyrinths in $\C^N$ intersected with $X$.
Labyrinths of this type were used by Globevnik in \cite{Globevnik2016MA}, but his construction 
is different from the one presented here and does not provide any control of the topology of the leaves. 
\end{remark}


\subsection*{Acknowledgements}
A.\ Alarc\'on is supported by the State Research Agency (SRA) and European Regional Development Fund (ERDF) via the grant no. MTM2017-89677-P, MICINN, Spain.
F.\ Forstneri\v c is supported  by the research program P1-0291 and the research grant 
J1-9104 from ARRS, Republic of Slovenia. 



\begin{thebibliography}{10}

\bibitem{Alarcon2018}
A.~Alarc{\'o}n.
\newblock {Complete complex hypersurfaces in the ball come in foliations}.
\newblock {\em ArXiv e-prints}, Feb. 2018.
\newblock \url{https://arxiv.org/abs/1802.02004}.

\bibitem{AlarconForstneric2013MA}
A.~Alarc{\'o}n and F.~Forstneri{\v{c}}.
\newblock Every bordered {R}iemann surface is a complete proper curve in a
  ball.
\newblock {\em Math. Ann.}, 357(3):1049--1070, 2013.

\bibitem{AlarconForstneric2014IM}
A.~Alarc{\'o}n and F.~Forstneri{\v{c}}.
\newblock Null curves and directed immersions of open {R}iemann surfaces.
\newblock {\em Invent. Math.}, 196(3):733--771, 2014.

\bibitem{AlarconForstnericJAMS}
A.~Alarc\'{o}n and F.~Forstneri\v{c}.
\newblock {New complex analytic methods in the theory of minimal surfaces: a
  survey}.
\newblock {\em J. Aust. Math. Soc.}, 106(3):287--341, 2019.

\bibitem{AlarconForstneric2019CY}
A.~Alarc\'{o}n and F.~Forstneri\v{c}.
\newblock {The Calabi-Yau problem for Riemann surfaces with finite genus and
  countably many ends}.
\newblock {\em arXiv e-prints}, Apr 2019.

\bibitem{AlarconGlobevnik2017}
A.~Alarc\'{o}n and J.~Globevnik.
\newblock Complete embedded complex curves in the ball of {$\Bbb C^2$} can have
  any topology.
\newblock {\em Anal. PDE}, 10(8):1987--1999, 2017.

\bibitem{AlarconGlobevnikLopez2016Crelle}
A.~{Alarc{\'o}n}, J.~{Globevnik}, and F.~J. {L{\'o}pez}.
\newblock {A construction of complete complex hypersurfaces in the ball with
  control on the topology}.
\newblock {\em J. Reine Angew. Math., {\rm to appear}}.
\newblock
  \url{https://www.degruyter.com/view/j/crll.ahead-of-print/crelle-2016-0061/crelle-2016-0061.xml}.

\bibitem{AlarconLopez2013MA}
A.~Alarc{\'o}n and F.~J. L{\'o}pez.
\newblock Null curves in {$\mathbb {C}^3$} and {C}alabi-{Y}au conjectures.
\newblock {\em Math. Ann.}, 355(2):429--455, 2013.

\bibitem{AlarconLopez2016JEMS}
A.~Alarc{\'o}n and F.~J. L{\'o}pez.
\newblock Complete bounded embedded complex curves in {$\mathbb {C}^2$}.
\newblock {\em J. Eur. Math. Soc. (JEMS)}, 18(8):1675--1705, 2016.

\bibitem{Drinovec2015JMAA}
B.~Drinovec~Drnov{\v s}ek.
\newblock Complete proper holomorphic embeddings of strictly pseudoconvex
  domains into balls.
\newblock {\em J. Math. Anal. Appl.}, 431(2):705--713, 2015.

\bibitem{Forstneric2003AM}
F.~Forstneri{\v{c}}.
\newblock Noncritical holomorphic functions on {S}tein manifolds.
\newblock {\em Acta Math.}, 191(2):143--189, 2003.

\bibitem{ForstnericRosay1993}
F.~Forstneri{\v{c}} and J.-P. Rosay.
\newblock Approximation of biholomorphic mappings by automorphisms of
  {${\mathbb C}^n$}.
\newblock {\em Invent. Math.}, 112(2):323--349, 1993.
\newblock {Erratum: {\it Invent. Math.}, 118(3):573--574, 1994}.

\bibitem{Forstneric2017E}
F.~Forstneri\v{c}.
\newblock {\em {Stein manifolds and holomorphic mappings. The homotopy
  principle in complex analysis (2nd edn).}}, volume~56 of {\em Ergebnisse der
  Mathematik und ihrer Grenzgebiete. 3. Folge}.
\newblock Berlin: Springer, 2017.

\bibitem{Globevnik2015AM}
J.~Globevnik.
\newblock A complete complex hypersurface in the ball of {${\mathbb C}^N$}.
\newblock {\em Ann. of Math. (2)}, 182(3):1067--1091, 2015.

\bibitem{Globevnik2016MA}
J.~Globevnik.
\newblock Holomorphic functions unbounded on curves of finite length.
\newblock {\em Math. Ann.}, 364(3-4):1343--1359, 2016.

\bibitem{Jones1979}
P.~W. Jones.
\newblock A complete bounded complex submanifold of {${\bf C}^{3}$}.
\newblock {\em Proc. Amer. Math. Soc.}, 76(2):305--306, 1979.

\bibitem{Stout2007}
E.~L. Stout.
\newblock {\em Polynomial convexity}, volume 261 of {\em Progress in
  Mathematics}.
\newblock Birkh\"auser Boston, Inc., Boston, MA, 2007.

\bibitem{Yang1977}
P.~Yang.
\newblock Curvature of complex submanifolds of {$C^{n}$}.
\newblock In {\em Several complex variables ({P}roc. {S}ympos. {P}ure {M}ath.,
  {V}ol. {XXX}, {P}art 2, {W}illiams {C}oll., {W}illiamstown, {M}ass., 1975)},
  pages 135--137. Amer. Math. Soc., Providence, R.I., 1977.

\end{thebibliography}


\vspace*{0.5cm}

\noindent Antonio Alarc\'{o}n

\noindent Departamento de Geometr\'{\i}a y Topolog\'{\i}a e Instituto de Matem\'aticas (IEMath-GR), Universidad de Granada, Campus de Fuentenueva s/n, E--18071 Granada, Spain

\noindent  e-mail: {\tt alarcon@ugr.es}

\vspace*{0.5cm}
\noindent Franc Forstneri\v c

\noindent Faculty of Mathematics and Physics, University of Ljubljana, Jadranska 19, SI--1000 Ljubljana, Slovenia

\noindent 
Institute of Mathematics, Physics and Mechanics, Jadranska 19, SI--1000 Ljubljana, Slovenia.

\noindent e-mail: {\tt franc.forstneric@fmf.uni-lj.si}

\end{document}